\documentclass[12pt,a4paper,draft,reqno]{amsproc}

\usepackage[cp1251]{inputenc}
\usepackage[english,russian]{babel}

\usepackage{amssymb}
\usepackage{a4wide}
\usepackage{mathrsfs,euscript}

\theoremstyle{plain}
\newtheorem{theorem}{\bf Теорема} 
\newtheorem{lemma}{\bf Лемма} 

\newtheorem{pro}{\bf Предложение}
\newtheorem{prob}{\bf Задача}

\theoremstyle{definition}

\theoremstyle{plain}
\newtoks\thehProclaim
\newtheorem*{Proclaim}{\the\thehProclaim}

\theoremstyle{definition}
\newtoks{\thehRemark}
\newtheorem*{Remark}{\the\thehRemark}

\DeclareRobustCommand{\No}{%
\ifmmode{\nfss@text{\textnumero}}\else\textnumero\fi}

\begin{document}

\sloppy


\large
УДК 517.53+517.95\\
\begin{center}
\textbf{МЕТОД ВОЗМУЩЕНИЙ ДЛЯ СИЛЬНО ЭЛЛИПТИЧЕСКИХ СИСТЕМ ВТОРОГО ПОРЯДКА\\ С ПОСТОЯННЫМИ КОЭФФИЦИЕНТАМИ\\}
\bigskip
\textbf{А.О. БАГАПШ\footnote{Работа выполнена при поддержке Российского научного фонда, проект 22-11-00071.}}
\end{center}

\normalsize
\bigskip

\noindent\textbf{Аннотация.} Рассмотрена классическая постановка задачи Дирихле для сильно эллиптической системы второго порядка с постоянными
коэффициентами в жордановых областях на плоскости. Показано, что решение задачи представляется в виде функционального ряда по степеням параметра,
определяющего отклонение оператора системы от лапласиана. Этот ряд сходится равномерно в замыкании области в предположении, что граница области и
заданная на ней граничная функция удовлетворяют достаточным условиям регулярности: композиция следа конформного отображения области на круг и
граничной функции принадлежит классу Гельдера с показателем больше, чем 1/2.

\noindent\textbf{Abstract.} The classical Dirichlet problem for a second-order strongly elliptic system with constant coefficients in a Jordan domain is considered.  We show that the solution of the problem can be represented as a functional series in powers of the parameter,
which determines the deviation of the system operator from the Laplacian. This series converges uniformly in the closure of the region under the assumption that the boundary of the region and the boundary function satisfy the sufficient regularity conditions: the trace of a conformal mapping of the domain onto a circle composed with the boundary function belongs to the Holder class with exponent greater than 1/2.

\noindent\textbf{Ключевые слова:} сильно эллиптическая система, задача Дирихле, метод возмущений.


\bigskip
\section{Введение}

В настоящей работе рассматривается система дифференциальных уравнений
\begin{equation}\label{el-sys}
\left(A\frac{\partial^2}{\partial x^2}+2B\frac{\partial^2}{\partial x\partial y}+C\frac{\partial^2}
{\partial y^2}\right)\left(\begin{matrix}u\\v\end{matrix}\right)=
\left(\begin{matrix}0\\0\end{matrix}\right)
\end{equation}
относительно вещественнозначных функций $u(x, y)$ и $v(x, y)$ вещественных переменных $x$ и $y$ с постоянными вещественными матрицами коэффициентов $A$, $B$, $C$ размера $2\times 2$. Изучаются системы такого вида, относящиеся к эллиптическому типу. Согласно определению
Петровского \cite{Petrovskiy-39}, это означает, что
\begin{equation*}
\det(A\xi^2+2B\xi\eta+C\eta^2)\ne 0\quad\text{при}\quad(\xi, \eta)\in\mathbb R^2\setminus(0, 0).
\end{equation*}

Введем комплекснозначную функцию $f=u+iv$ комплексного переменного $z=x+iy$ и оператор системы \eqref{el-sys}
\begin{equation*}
Lf=\left(A\frac{\partial^2}{\partial x^2}+2B\frac{\partial^2}{\partial x\partial y}+C\frac{\partial^2}
{\partial y^2}\right)\left(\begin{matrix}u\\v\end{matrix}\right).
\end{equation*}
Классическая постановка задачи Дирихле для такого оператора в жордановой области формулируется следующим образом.
\begin{prob}\label{dirprob}
Пусть $\varOmega$ --- жорданова область с границей $\varGamma$. Для заданной граничной функции $h\in C(\varGamma)$ найти такую функцию
$f\in C(\overline\varOmega)$, что $Lf=0$ в $\varOmega$ и $f|_\varGamma=h$.
\end{prob}

Изучение вопроса о разрешимости задачи Дирихле привело к выделению подкласса сильно эллиптических систем, которые были определены несколькими
способами; мы будем пользоваться определением из \cite{HuaLinWu-65}, согласно которому требуется
\begin{equation*}
\det(A\alpha+2\beta B+\gamma C)\ne 0\quad\text{при}\quad\beta^2-\alpha\gamma<0.
\end{equation*}
Оно эквивалентно хорошо известному определению Вишика \cite{Vishik-51}.

Вид системы \eqref{el-sys}, ее принадлежность к классу эллиптических или сильно эллиптических систем сохраняются при трех классах невырожденных
преобразований: 1) линейной замены переменных $(x, y)$; 2) линейной замены искомых функций $(u, v)$; 3) линейной комбинации уравнений системы.
При этом специально подобранная серия таких преобразований с последующим сложением первого из полученных уравнений со вторым, умноженным на мнимую
единицу $i$, позволяет привести любую эллиптическую систему \eqref{el-sys} к комплексному уравнению
\begin{equation}\label{canon}
(\partial\overline\partial+\tau\partial^2)g(z)+\sigma(\tau\partial\overline\partial+\partial^2)\overline{g(z)}=0
\end{equation}
относительно комплекснозначной функции $g$ комплексного переменного $z=x+iy$ со всего двумя параметрами $\tau\in[0,1)$ и
$\sigma\in[0,1)\cup(1, \infty]$ (см. \cite{HuaLinWu-85}, \cite{BagFed-17}). Здесь
\begin{equation*}
\partial=\frac{\partial}{\partial z}=\frac{1}{2}\left(\frac{\partial}{\partial x}-i\frac{\partial}{\partial y}\right),\qquad
\overline\partial=\frac{\partial}{\partial\overline z}=\frac{1}{2}\left(\frac{\partial}{\partial x}+i\frac{\partial}{\partial y}\right)
\end{equation*}
--- операторы Коши --- Римана в новой системе координат. При $\sigma=\infty$ считаем, что уравнение \eqref{canon} приобретает вид
\begin{equation*}
(\tau\partial\overline\partial+\partial^2)\overline{g(z)}=0.
\end{equation*}
В случае сильной эллиптичности будет $\sigma\in[0, 1)$.

Перечислим несколько хорошо известных частных случаев уравнения \eqref{canon}. При $\tau=\sigma=0$ имеем комплексное уравнение Лапласа
$\Delta g(z)=4\partial\overline\partial g(z)=0$, а при $\tau=0$, $\sigma=\infty$ --- уравнение Бицадзе \cite{Bitsadze-48}
$\overline\partial^2 g(z)=0$. Если $\tau=0$, то возникает плоское изотропное уравнение Ламе теории упругости, записанное в комплексном виде $\partial\overline\partial g(z)+\sigma\partial^2\overline{f(z)}=0$, причем параметр $\sigma$ связан с коэффициентом Пуассона $p$ соотношением $\sigma=1/(3-4p)$. Поскольку, как известно \cite{LanLif-87}, $p\in(0, 1/2)$, а значит, $\sigma\in(1/3, 1)$, то соответствующая система \eqref{canon} сильно эллиптическая. Если же $\sigma=0$, то получаем систему, называемую кососимметрической, которая может быть записана в виде уравнения $ag_{xx}+2bg_{xy}+cg_{yy}=0$  с комплексными коэффициентами $a$, $b$, $c$.

Уравнение \eqref{canon} представляет собой возмущенное по двум параметрам $\tau$ и $\sigma$ уравнение Лапласа, причем в случае сильной
эллиптичности эти параметры относительно малы: $\tau, \sigma\in[0, 1)$. Чтобы подчеркнуть данное обстоятельство, сделаем еще одно, последнее
преобразование над уравнением \eqref{canon}, отделив лапласиан от остальной его части:
\begin{equation}\label{canon-2}
\partial\overline\partial(\mathcal T_{1, \sigma\tau}g)+\partial^2(\mathcal T_{\tau, \sigma}g)=0,
\end{equation}
где используется оператор аффинного преобразования
\begin{equation}\label{T-oper}
\mathcal T_{\alpha,\beta}:=\alpha\mathcal I+\beta\mathcal C,
\end{equation}
с (вообще говоря) комплексными параметрами $\alpha$ и $\beta$, выражающийся через тождественный оператор $\mathcal I\colon w\to w$  и оператор
комплексного сопряжения $\mathcal C\colon w\to\overline w$.  Если $|\alpha|\ne|\beta|$, то существует обратный оператор
$$
\mathcal T_{\alpha, \beta}^{-1}=\frac{1}{|\alpha|^2-|\beta|^2}\mathcal T_{\overline\alpha,-\beta}.
$$
Норма оператора \eqref{T-oper} как отображения $\mathbb C\to\mathbb C$ равна $\|\mathcal T_{\alpha, \beta}\|=|\alpha|+|\beta|$.

Считая уравнение \eqref{canon-2} сильно эллиптическим, заменим в нем искомую функцию $g$ на $f=\mathcal T_{1, \sigma\tau}g$ (невырожденным в этом случае преобразованием) и перепишем \eqref{canon-2} в виде
\begin{equation}\label{el-sys-pert}
\mathcal Lf:=\partial\overline\partial f+\partial^2(Tf)=0,
\end{equation}
где
$$
T=\mathcal T_{\tau, \sigma}\mathcal T_{1, \sigma\tau}^{-1}=\frac{\tau(1-\sigma^2)\mathcal I+\sigma(1-\tau^2)\mathcal C}{1-\sigma^2\tau^2}.
$$
Полученное уравнение \eqref{el-sys-pert} представляет собой уравнение Лапласа, возмущенное по оператору $T$ с нормой
\begin{equation*}
\|T\|=\frac{\tau+\sigma}{1+\sigma\tau},
\end{equation*}
которая в рассматриваемом сильно эллиптическом случае, когда $\tau, \sigma\in[0, 1)$, оказывается меньше единицы. Введем нормированный на единицу
оператор
\begin{equation*}
T_0=\|T\|^{-1}T=\mathcal T_{\alpha_0, \beta_0},
\end{equation*}
где
\begin{equation*}
\alpha_0=\frac{\tau(1-\sigma^2)}{(\tau+\sigma)(1-\sigma\tau)},\qquad\beta_0=\frac{\sigma(1-\tau^2)}{(\tau+\sigma)(1-\sigma\tau)},
\end{equation*}
и перепишем с его помощью \eqref{el-sys-pert} в окончательном виде
\begin{equation}\label{el-syst-pert-2}
\mathcal Lf=\partial\overline\partial f+\|T\|\partial^2(T_0f)=0.
\end{equation}
В уравнении \eqref{el-syst-pert-2} малым параметром является $\|T\|<1$.

\section{Метод возмущений}

Для решения задачи Дирихле применим метод возмущения по величине $\|T\|$, который состоит в поиске решения $f$ в виде ряда
\begin{equation}\label{series}
f=\sum\limits_{n=0}^\infty f_n\|T\|^n,
\end{equation}
в котором функции $f_n$ находятся с помощью подстановки разложения \eqref{series} в уравнение $\mathcal Lf=0$ и приравнивания к нулю множителей
при одинаковых степенях величины $\|T\|$; при этом полагаем $f_0|_\varGamma=h$ и $f_n|_\varGamma=0$ для $n\ge 1$.

Таким образом, получаем следующие краевые задачи для последовательного отыскания функций $f_n$:
\begin{equation}\label{f0-prob}
\partial\overline\partial f_0=0\quad\text{в}\quad\varOmega,\qquad\qquad f_0|_\varGamma=h
\end{equation}
и
\begin{equation}\label{fn-prob}
\partial\overline\partial f_n=-\partial^2(T_0 f_{n-1})\quad\text{в}\quad\varOmega,\qquad\qquad f_n|_\varGamma=0
\end{equation}
для $n\geqslant 1$.

Пусть $\omega\colon\mathbb D\to\varOmega$ --- некоторое конформное отображение единичного круга $\mathbb D:=\{z\in\mathbb C\colon|z|<1\}$ на область $\varOmega$. В случае жордановой области $\varOmega$ по теореме Каратеодори отображение $\omega$ продолжается до гомеоморфизма замкнутых
областей $\overline{\mathbb D}$ и $\overline\varOmega$ и т.е. $\omega\in C(\overline{\mathbb D})$. Для дальнейшего удобно перенести задачи
\eqref{f0-prob} и \eqref{fn-prob} в круг $\mathbb D$ с помощью введенного конформного отображения. Пусть
\begin{equation*}
F=f\circ\omega,\qquad H=h\circ\omega\qquad F_n=F_n\circ\omega.
\end{equation*}
Тогда
\begin{equation}\label{oper-change}
\mathcal Lf=\frac{1}{|\omega'|^2}\left[\partial\overline\partial F+\|T\|\partial\left(\frac{\overline{\omega'}}{\omega'}\partial(T_0 F)\right)\right]=:\mathcal MF.
\end{equation}
Из \eqref{series}, \eqref{f0-prob} и \eqref{fn-prob} следует, что
\begin{equation}\label{series-disk}
F=\sum\limits_{n=0}^\infty F_n\|T\|^n,
\end{equation}
где
\begin{equation}\label{F0-prob}
\partial\overline\partial F_0=0\quad\text{в}\quad\mathbb D,\qquad\qquad F_0|_{\mathbb T}=H
\end{equation}
и
\begin{equation}\label{Fn-prob}
\partial\overline\partial F_n=-\partial\left(\frac{\overline{\omega'}}{\omega'}\partial(T_0 F_{n-1})\right)\quad\text{в}\quad\varOmega,\qquad\qquad F_n|_{\mathbb T}=0
\end{equation}
для $n\geqslant 1$. В случае достаточной регулярности функции $F_{n-1}$ при фиксированном номере $n$ можно с помощью функции Грина
\begin{equation}\label{Green-func}
G(\zeta, z)=\frac{2}{\pi}\log\left|\frac{\zeta-z}{1-\zeta\overline z}\right|
\end{equation}
для оператора $\partial\overline\partial$ в круге $\mathbb D$ записать решения задач \eqref{F0-prob} и \eqref{Fn-prob}:
\begin{equation}\label{induct-disk}
F_0(z)=\frac{1}{2i}\int_{\mathbb T}\partial_\zeta G(\zeta, z)h(\zeta)d\zeta,\qquad
F_n(z)=-\int_{\mathbb D} G(\zeta, z)\partial\left(\frac{\overline{\omega'(\zeta)}}{\omega'(\zeta)}\partial(T_0 F_{n-1}(\zeta))\right)d\mu,
\end{equation}
$n\geqslant 1$, где $\mu$ --- мера Лебега.

Определим операторы
\begin{equation*}\label{ind-oper}
\mathcal P[\varphi(z)]:=\frac{1}{2i}\int_{\mathbb T}\partial_\zeta G(\zeta, z)\varphi(\zeta)d\zeta,\qquad
\mathcal K[\varphi(z)]:=\int_{\mathbb D}\partial_\zeta G(\zeta, z)\varphi(\zeta)d\mu
\end{equation*}
и
\begin{equation*}\label{der-oper}
\mathcal K_\partial[\varphi(z)]:=\text{p.v.}\int_{\mathbb D}\partial_z\partial_\zeta G(\zeta, z)\varphi(\zeta)d\mu,\qquad
\mathcal K_{\overline\partial}[\varphi(z)]:=\text{p.v.}\int_{\mathbb D}\partial_{\overline z}\partial_\zeta G(\zeta, z)\varphi(\zeta)d\mu,
\end{equation*}
из которых $\mathcal P$ задается на классе функций $C(\mathbb T)$, а остальные на $L_p(\mathbb D)$, причем последние два интеграла понимаются в
смысле главного значения. В дальнейшем обозначение $\text{v.p.}$ будем для краткости опускать. Из формулы \eqref{Green-func} получаем
\begin{equation}\label{ind-oper-2}
\begin{aligned}
\mathcal P[\varphi(z)]=\frac{1}{2\pi}\int_{\mathbb T}\left(\frac{1}{\zeta-z}+\frac{\overline z}{1-\zeta\overline z}\right)\varphi(\zeta)d\zeta,\\
\mathcal K[\varphi(z)]=\frac{1}{\pi}\int_{\mathbb D}\left(\frac{1}{\zeta-z}+\frac{\overline z}{1-\zeta\overline z}\right)\varphi(\zeta)d\mu
\end{aligned}
\end{equation}
и
\begin{equation}\label{der-oper-2}
\mathcal K_\partial[\varphi(z)]=\frac{1}{\pi}\int_{\mathbb D}\frac{\varphi(\zeta)d\mu}{(\zeta-z)^2},\qquad
\mathcal K_{\overline\partial}[\varphi(z)]=\frac{1}{\pi}\int_{\mathbb D}\frac{\varphi(\zeta)d\mu}{(1-\zeta\overline z)^2}-\varphi(z).
\end{equation}
С помощью введенных операторов формулы \eqref{induct-disk} для построения функций $F_n$ можно записать в виде
\begin{equation}\label{induct-oper}
F_0=\mathcal P[h],\qquad F_n=\mathcal K[(\overline{\omega'}/\omega')\partial(T_0 F_{n-1})],\quad n\geqslant 1.
\end{equation}
При этом
\begin{equation*}
\partial F_n=\mathcal K_\partial[(\overline{\omega'}/\omega')\partial(T_0 F_{n-1})],\qquad
\overline\partial F_n=\mathcal K_{\overline\partial}[(\overline{\omega'}/\omega')\partial(T_0 F_{n-1})],\quad n\geqslant 1.
\end{equation*}

Введем также обозначения для частичных сумм рядов \eqref{series} и \eqref{series-disk} соответственно:
\begin{equation}\label{part-sum}
s_m=\sum_{n=0}^m f_n\|T\|^n,\qquad S_m=\sum_{n=0}^m F_n\|T\|^n.
\end{equation}

В настоящей работе доказывается следующая теорема сходимости.

\begin{theorem}\label{dirprob-solv}
Пусть жорданова область $\varOmega$ и заданная на ее границе $\varGamma$ функция $h$ таковы, что $h\circ\omega\in C^\alpha(\mathbb T)$ при $1/2<\alpha<1$, где $\omega$ --- некоторое конформное отображение единичного круга $\mathbb D$ на $\varOmega$. Тогда при любом значении $\|T||\in[0, 1)$ ряд \eqref{series} с функциями $f_n=F_n\circ\omega^{-1}$, где $F_n$ заданы согласно \eqref{induct-oper}, сходится в норме пространства $C(\overline\varOmega)$ к функции $f\in C(\overline\varOmega)$, удовлетворяющей в $\varOmega$ уравнению $\mathcal Lf=0$ и совпадающей на $\varGamma$ с $h$.
\end{theorem}

Условие $h\circ\omega\in C^\alpha(\mathbb T)$, $\alpha\in(1/2, 1)$, выполняется, например, при $h\in C^\beta(\Gamma)$ и
$\omega\in C^\gamma(\mathbb T)$, где $\beta\gamma=\alpha\in(1/2, 1)$. Действительно, в этом случае
$$
|h\circ\omega(z_1)-h\circ\omega(z_2)|\leqslant[h]_\alpha|\omega(z_1)-\omega(z_2)|^\beta\leqslant[h]_\beta[\omega|_{\mathbb T}]_\gamma^\beta
|z_1-z_2|^{\beta\gamma},
$$
где $[\varphi]_\alpha:=\sup_{\zeta_1\ne\zeta_2}|\varphi(\zeta_1)-\varphi(\zeta_2)|/|\zeta_1-\zeta_2|^\alpha$.

Теорема \ref{dirprob-solv} является распространением аналогичного результата, полученного в работе автора \cite{Bagapsh-2023-CVEE} для
кососимметрической сильно эллиптической системы, являющейся частным случаем рассматриваемой здесь системы, отвечающим значению параметра
$\sigma=0$.

Отметим, что не все рассматриваемые здесь сильно эллиптические системы \eqref{canon} обладают функционалом энергии, с помощью которого возможна
вариационная переформулировка задачи Дирихле, стоящая за доказательством теоремы Лебега  об общей разрешимости задачи Дирихле для уравнения Лапласа в односвязной области, см. \cite{Lebesgue-07}. Система канонического вида \eqref{canon} обладает функционалом энергии в виде интеграла по
области от квадратичной формы первых производных только при соотношении параметров $\sigma>\tau$; такие системы называются симметризуемыми, см.
\cite{BagFed-2022}. Это обстоятельство является причиной того, что вопрос о разрешимости задачи Дирихле для общих сильно эллиптических систем вида \eqref{el-sys} в односвязных или хотя бы жордановых областях с произвольными непрерывными граничными данными является открытым.

В настоящее время наибольшим продвижением в вопросе о разрешимости задачи \ref{dirprob} является результат Веркоты и Фогеля
\cite{VerVog-97}, устанавливающий общую разрешимость задачи Дирихле в областях с кусочно гладкими липшицевыми границами при произвольных непрерывных граничных данных. В доказываемой здесь теореме \ref{dirprob-solv} граничные функции берутся из более узкого класса Гельдера,
однако область может принадлежать более широкому классу по сравнению с \cite{VerVog-97}.




\section{Доказательство сходимости метода возмущений}

\begin{lemma}\label{Pois-Wp-lem}
Если $\varphi\in C^\alpha(\mathbb T)$, где $1/2<\alpha<1$, то $\mathcal P[\varphi]\in W^1_p(\mathbb D)$ с любым показателем $0<p<(2(1-\alpha))^{-1}$.
\end{lemma}

\begin{proof}
Пусть $\psi=\mathcal P\varphi$. Поскольку $\varphi\in C(\mathbb T)$, то по свойству интеграла Пуассона, $\psi\in C(\overline{\mathbb D})$, так что заведомо $\psi\in L_p(\mathbb D)$. Докажем $L_p$--интегрируемость первых производных. Представим функцию $\psi$ в виде суммы $\psi(z)=\psi_1(z)+\psi_2(z)$ голоморфных компонент
\begin{equation*}
\psi_1(z)=\frac{1}{2\pi i}\int_{\mathbb T}\frac{\varphi(\zeta)d\zeta}{\zeta-z},\qquad
\psi_2(z)=\frac{1}{2\pi i}\int_{\mathbb T}\frac{\varphi(\zeta)d\overline\zeta}{\overline\zeta-\overline z}-\frac{1}{2\pi}
\int_{\mathbb T}\varphi(\zeta)|d\zeta|.
\end{equation*}
В силу теоремы Привалова \cite{Privalov-39} для интеграла типа Коши, из принадлежности $\varphi\in C^\alpha(\mathbb T)$ при $1/2<\alpha<1$
следует, что $\psi_1\in C^{2\alpha-1}(\overline{\mathbb D})$. Обозначим $T(z, r):=\{\zeta\in\mathbb C\colon |\zeta-z|=r\}\subset\mathbb D$. Из
формулы Коши
$$
\psi_1(z)=\frac{1}{2\pi i}\int_{T(z,r)}\frac{\psi_1(\zeta)d\zeta}{\zeta-z}
$$
находим
\begin{equation*}
\partial\psi(z)=\psi_1'(z)=\frac{1}{2\pi i}\int_{T(z,r)}\frac{\psi_1(\zeta)d\zeta}{(\zeta-z)^2}=
\frac{1}{2\pi i}\int_{T(z,r)}\frac{\psi_1(\zeta)-\psi_1(z)}{(\zeta-z)^2}d\zeta,
\end{equation*}
откуда выводим оценку
$$
|\partial\psi(z)|\leqslant\frac{1}{2\pi i}\int_{T(z,r)}\frac{|\psi_1(\zeta)-\psi_1(z)|}{|\zeta-z|^2}|d\zeta|\leqslant
\frac{1}{2\pi i}\int_{T(z,r)}\frac{c_\alpha|\zeta-z|^{2\alpha-1}}{|\zeta-z|^2}|d\zeta|=\frac{c_\alpha}{r^{2(1-\alpha)}},
$$
где $c_\alpha=\sup_{\zeta\ne z}|\psi_1(\zeta)-\psi_1(z)|/|\zeta-z|^{2\alpha-1}$. Предельным переходом $r\to(1-|z|)$ получаем
$|\partial\psi(z)|\leqslant c_\alpha(1-|z|)^{2(\alpha-1)}$, см. также \cite[стр. 74]{Duren-70} или \cite[стр. 50]{Pommerenke-92}. Это означает, что
$\partial\psi\in L_p(\mathbb D)$, если $2(1-\alpha)p<1$. Аналогичным образом устанавливается $L_p$--интегрируемость производной $\overline\partial\psi=\psi_2'$ при том же условии на $p$. Лемма доказана.
\end{proof}

Рассмотрим оператор Берлинга
\begin{equation}\label{polar-oper}
\mathcal B\varphi(z):=\frac{1}{\pi}\int_{\mathbb C}\frac{\varphi(\zeta)d\mu}{(\zeta-z)^2},
\end{equation}
осуществляющий, по теореме Кальдерона --- Зигмунда \cite{CalZig-52}, ограниченное отображение пространства $L_p(\mathbb C)$ в себя при любом $p\in(1, \infty)$. Обозначим через $\|\mathcal B\|_p$ его норму как отображения $L_p(\mathbb C)\to L_p(\mathbb C)$, аналогичным образом будем обозначать нормы операторов, действующих в $L_p(U)$ для произвольной области $U$. Для дальнейшего является существенным то обстоятельство, что $\|\mathcal B\|_p\to 1$ при $p\to 2$, см. \cite[стр. 89]{Ahlfors-66}, \cite[стр. 5--6]{Christ-90}.

\begin{pro}\label{Lp-oper-pro}
Операторы \eqref{ind-oper-2}, \eqref{der-oper-2} обладают следующими свойствами:

\noindent (i) $\mathcal K\colon L_p(\mathbb D)\to L_p(\mathbb D)$ ограничен при $p>1$;

\noindent (ii) $\mathcal K_\partial\colon L_p(\mathbb D)\to L_p(\mathbb D)$ ограничен при $p>1$, причем $\|\mathcal K_\partial\|_p=\|\mathcal B\|_p\to 1$
при $p\to 2$;

\noindent (iii) $\mathcal K_{\overline\partial}\colon L_p(\mathbb D)\to L_p(\mathbb D)$ ограничен при $p>1$, причем
$\|\mathcal K_{\overline\partial}\|_p\to 1$ при $p\to 2$.
\end{pro}

\begin{proof}
(i) вытекает из того, что ядро интегрального оператора $\mathcal K$ состоит из суммы двух ядер со слабой особенностью.

(ii) Пусть $\varphi\in L_p(\mathbb D)$, $p>1$. Тогда $\mathcal K_\partial\,\varphi=\mathcal B\varphi_1$, где функция $\varphi_1$ совпадает с $\varphi$ в круге $\mathbb D$ и равна нулю вне $\mathbb D$, так что $\|\varphi_1\|_{L_p(\mathbb C)}=\|\varphi\|_{L_p(\mathbb D)}$. Следовательно,
$\|\mathcal K_\partial\|_p=\|\mathcal B\|_p$.

(iii) Устроим в интеграле для $\mathcal K_{\overline\partial}$ из формулы \eqref{der-oper-2} замену переменной $\zeta$ на $\xi=1/\overline\zeta$ и получим
\begin{equation}\label{iii-proof}
\mathcal K_{\overline\partial}[\varphi(z)]=\frac{1}{\pi}\int_{\mathbb C\setminus\overline{\mathbb D}}\frac{\varphi(1/\overline\xi)}{\xi^2}\cdot
\frac{d\mu}{(\overline\xi-\overline z)^2}-\varphi(z)=\overline{\mathcal B[\varphi_2(z)]}-\varphi(z),
\end{equation}
где $\varphi_2(z)=\overline{\varphi(1/\overline z)}/\overline z^2$ при $z\in\mathbb C\setminus\overline{\mathbb D}$ и $\varphi_2(z)=0$ при $z\in\mathbb D$.
Устраивая обратную замену $\zeta=1/\overline\xi$, находим
$$
\|\varphi_2\|_{L_p(\mathbb D)}=\left(\int_{\mathbb C\setminus\overline{\mathbb D}}|\varphi_2(\xi)|^p d\mu\right)^{\frac{1}{p}}=
\left(\int_{\mathbb D}|\xi|^{2p-4}\cdot|\varphi(\zeta)|^p d\mu\right)^{\frac{1}{p}}\leqslant\|\varphi\|_{L_p(\mathbb D)}
$$
при $p\geqslant 2$ с равенством при $p=2$. Тогда из \eqref{iii-proof} следует, что
$$
\|\mathcal K_{\overline\partial}\,\varphi\|_{L_p(\mathbb D)}\leqslant(\|\mathcal B\|_p+1)\|\varphi\|_{L_p(\mathbb D)},
$$
т.е. $\mathcal K_{\overline\partial}\colon L_p(\mathbb D)\to L_p(\mathbb D)$ ограничен при $p>1$.

Найдем $\|\mathcal K_{\overline\partial}\|_2$. Пусть сначала $\varphi$ --- произвольная пробная функция из класса $C^2_0(\mathbb D)$ дважды непрерывно дифференцируемых в $\mathbb C$ функций с компактным носителем $\text{supp}(\varphi)\subset D_r:=\{|z|<r\}$, где $r\in(0, 1)$.
Тогда $\mathcal K[\varphi(z)]\in C(\overline{\mathbb D})$, см. \cite[стр. 85]{Ahlfors-66}. Из \eqref{ind-oper-2} имеем
$$
\mathcal K[\varphi(z)]=\frac{1-|z|^2}{\pi}\int_{\text{supp}(\varphi)}\frac{\varphi(\zeta)d\mu(\zeta)}{(\zeta-z)(1-\zeta\overline z)},
$$
откуда с помощью неравенства Гельдера находим при $|z|>r$
\begin{equation}\label{K-vanish}
|\mathcal K[\varphi(z)]|\leqslant\frac{1-|z|^2}{\pi}\sqrt{\int_{D_r}\frac{|\varphi(\zeta)|^2d\mu(\zeta)}{|\zeta-z|^2}
\int_{D_r}\frac{|\varphi(\zeta)|^2d\mu(\zeta)}{|1-\zeta\overline z|^2}}\leqslant
\frac{(1-|z|^2)\|\varphi\|_{L_p(\mathbb D)}^2}{\pi(|z|-r)(1-r|z|)}\to 0
\end{equation}
при $|z|\to 1$. Поскольку $\varphi(z)=0$ и $\mathcal K[\varphi(z)]=0$ при $z\in\mathbb T$, то, применяя несколько раз интегрирование по частям, получаем
\begin{multline*}
\|\mathcal K_{\overline\partial}\,\varphi\|_{L_2(\mathbb D)}^2=\int_{\mathbb D}\frac{\partial}{\partial\overline z}\mathcal K[\varphi(z)]
\cdot\frac{\partial}{\partial z}\overline{\mathcal K[\varphi(z)]}d\mu(z)=-\int_{\mathbb D}\overline{\mathcal K[\varphi(z)]}
\frac{\partial^2}{\partial z\partial\overline z}\mathcal K[\varphi(z)]d\mu(z)=\\
=\int_{\mathbb D}\overline{\mathcal K[\varphi(z)]}\frac{\partial\varphi(z)}{\partial z}d\mu(z)=
-\int_{\mathbb D}\varphi(z)\frac{\partial}{\partial z}\overline{\mathcal K[\varphi(z)]}d\mu(z)=\\
=-\int_{\mathbb D}\varphi(z)\left(\frac{1}{\pi}\int_{\mathbb D}\frac{\overline{\varphi(\zeta)}d\mu(\zeta)}{(1-\overline\zeta z)^2}
-\overline{\varphi(z)}\right)d\mu(z)=\|\varphi\|_{L_2(\mathbb D)}^2-
\frac{1}{\pi}\int_{\mathbb D}\int_{\mathbb D}\frac{\varphi(z)\overline{\varphi(\zeta)}}{(1-\overline\zeta z)^2}d\mu(\zeta)d\mu(z).
\end{multline*}
Вычитаемая из $\|\varphi\|_{L_2(\mathbb D)}^2$ величина равна
\begin{multline*}
\frac{1}{\pi}\int_{\mathbb D}\int_{\mathbb D}\frac{\varphi(z)\overline{\varphi(\zeta)}}{(1-\overline\zeta z)^2}d\mu(\zeta)d\mu(z)=\\
=\sum\limits_{n=0}^\infty(n+1)\frac{1}{\pi}\int_{\mathbb D}\int_{\mathbb D}\varphi(z)\overline{\varphi(\zeta)}z^n\overline\zeta^n
d\mu(\zeta)d\mu(z)=\\
=\sum\limits_{n=0}^\infty(n+1)\frac{1}{\pi}\left|\int_{\mathbb D}\varphi(z)z^n d\mu(z)\right|^2\geqslant 0,
\end{multline*}
поэтому
$$
\|\mathcal K_{\overline\partial}\,\varphi\|_{L_2(\mathbb D)}\leqslant\|\varphi\|_{L_2(\mathbb D)}
$$
с равенством на функциях $\varphi\in C^2_0(\mathbb D)$, для которых $\int_{\mathbb D}\varphi(z)z^n d\mu(z)=0$ при $n=0, 1, \dots$ Приближая
функции из $L_p(\mathbb D)$ функциями класса $C^2_0(\mathbb D)$, получим ту же самую оценку. Таким образом,
$\|\mathcal K_{\overline\partial}\|_2=1$. Поскольку норма $\|\mathcal K_{\overline\partial}\|_p$ существует при всех $p>1$, то из теоремы
М. Рисса --- Торина \cite[стр. 113]{Ahlfors-66}, согласно которой величина $\log\|\mathcal K_{\overline\partial}\|_p$ является выпуклой функцией
переменного $1/p$, вытекает непрерывность этой величины относительно $p$. Следовательно, $\|\mathcal K_{\overline\partial}\|_p\to 1$ при $p\to 2$.
Предложение доказано.
\end{proof}

\begin{proof}[Доказательство теоремы \ref{dirprob-solv}]
\textbf{Шаг 1.} Установим сначала сходимость ряда \eqref{series-disk} вместе с первыми частными производными в норме пространства $L_p(\mathbb D)$. Из леммы \ref{Pois-Wp-lem} следует, что функция $F_0=\mathcal P[h]$ принадлежит пространству Соболева $W^1_p(\mathbb D)$ при $p<(2(1-\alpha))^{-1}$. Предположим, что $F_{n-1}\in L_p(\mathbb D)$, $p>2$, при некотором номере $n$. Тогда
$$
\partial F_n=\mathcal K_\partial[(\overline{\omega'}/\omega')(\alpha_0\partial F_{n-1}+\beta_0\partial\overline{F_{n-1}})],\qquad
\overline\partial F_n=\mathcal K_{\overline\partial}[(\overline{\omega'}/\omega')(\alpha_0\partial F_{n-1}+\beta_0\partial\overline{F_{n-1}})]
$$
в смысле распределений (см. \cite[стр. 90]{Ahlfors-66}). Используя эти соотношения и применяя предложение \ref{Lp-oper-pro}, а также тот факт, что
$|\alpha_0|+|\beta_0|=1$, выводим из формулы \eqref{induct-oper} оценки
$$
\begin{aligned}
&\|\partial F_n\|_{L_p(\mathbb D)}
\leqslant\|\mathcal K_\partial\|_p\max\{\|\partial F_{n-1}\|_{L_p(\mathbb D)}, \|\overline\partial F_{n-1}\|_{L_p(\mathbb D)}\},\\
&\|\overline\partial F_n\|_{L_p(\mathbb D)}
\leqslant\|\mathcal K_{\overline\partial}\|_p\max\{\|\partial F_{n-1}\|_{L_p(\mathbb D)}, \|\overline\partial F_{n-1}\|_{L_p(\mathbb D)}\},
\end{aligned}
$$
из которых следует
\begin{equation}\label{der-Fn-est}
\|DF_n\|_{L_p(\mathbb D)}\leqslant\|\mathcal{DK}\|_p\cdot\|DF_{n-1}\|_{L_p(\mathbb D)},
\end{equation}
где
\begin{equation*}
\|DF_n\|_{L_p(\mathbb D)}:=\max\{\|\partial F_n\|_{L_p(\mathbb D)}, \|\overline\partial F_n\|_{L_p(\mathbb D)}\},\qquad
\|\mathcal{DK}\|_p:=\max\{\|\mathcal K_\partial\|_p, \|\mathcal K_{\overline\partial}\|_p\}.
\end{equation*}
Тогда отсюда и из \eqref{induct-oper} получаем
\begin{equation}\label{Fn-est}
\|F_n\|_{L_p(\mathbb D)}\leqslant\|\mathcal K\|_p\cdot\|DF_{n-1}\|_{L_p(\mathbb D)}\leqslant
\|\mathcal K\|_p\cdot\|\mathcal{DK}\|_p^{n-1}\cdot\|DF_0\|_{L_p(\mathbb D)}.
\end{equation}

Оценка \eqref{Fn-est} доказывает сходимость при $\|\mathcal{DK}\|_p\cdot\|T\|<1$ ряда \eqref{series-disk} в норме $L_p(\mathbb D)$ к своей сумме
$F\in L_p(\mathbb D)$, причем
\begin{multline}\label{F-est}
\|F\|_{L_p(\mathbb D)}=\left\|\sum_{n=0}^\infty F_n\|T\|^n\,\right\|_{L_p(\mathbb D)}\leqslant\|F_0\|_{L_p(\mathbb D)}+
\sum_{n=1}^\infty\|F_n\|_{L_p(\mathbb D)}\cdot\|T\|^n\leqslant\\
\leqslant\|F_0\|_{L_p(\mathbb D)}+\sum_{n=1}^\infty\|\mathcal K\|_p\cdot\|\mathcal{DK}\|_p^{n-1}\cdot\|DF_0\|_{L_p(\mathbb D)}\cdot\|T\|^n=\\
=\|F_0\|_{L_p(\mathbb D)}+\frac{\|\mathcal K\|_p\cdot\|T\|}{1-\|\mathcal{DK}\|_p\cdot\|T\|}\|DF_0\|_{L_p(\mathbb D)}.
\end{multline}
Оценка \eqref{der-Fn-est} доказывает, что, кроме того, и первые частные производные ряда \eqref{series-disk} сходятся в той же норме к соответствующим производным функции $F$:
\begin{equation}\label{der-F-est}
\|DF\|_{L_p(\mathbb D)}\leqslant\frac{\|DF_0\|_{L_p(\mathbb D)}}{1-\|\mathcal{DK}\|_p\cdot\|T\|},
\end{equation}
где $\|DF\|_{L_p(\mathbb D)}:=\max\{\|\partial F\|_{L_p(\mathbb D)}$. Полученные оценки \eqref{F-est} и \eqref{der-F-est} означают сходимость ряда \eqref{series-disk} в норме пространства Соболева $W_p^1(\mathbb D)$:
\begin{equation}\label{ser-conv}
\lim\limits_{m\to\infty}\|F-S_m\|_{W_p^1(\mathbb D)}=0.
\end{equation}
По теореме вложения Соболева \cite{Sobolev-88}, $W_p^1(\mathbb D)\subset C(\overline{\mathbb D})$ при $1-2/p>0$, или $p>2$, причем вложение
компактно. Следовательно, поскольку $F\in W_p^1(\mathbb D)$, то $F\in C(\overline{\mathbb D})$ и тогда
$f=F\circ\omega^{-1}\in C(\overline\varOmega)$. В силу компактности вложения, ряд \eqref{series-disk}, а следовательно, и \eqref{series}, сходятся равномерно в $\overline{\mathbb D}$ и $\overline\varOmega$ соответственно.

\textbf{Шаг 2.} Теперь докажем, что функция $f=F\circ\omega^{-1}$ удовлетворяет уравнению $\mathcal Lf=0$ в области $\varOmega$, установив сначала выполнение этого равенства в обобщенном смысле.

Пусть $\phi$ --- произвольная пробная функция из класса $C^2_0(\mathbb D)$ и
$$
\langle g|\phi\rangle:=\int_{\mathbb C} g(z)\phi(z)d\mu
$$
есть действие обобщенной функции $g$ на функцию $\phi$. Имеем
\begin{multline*}
\langle F_n|\partial\overline\partial\phi\rangle=\int_{\mathbb D}\partial\overline\partial\phi(z)d\mu(z)
\int_{\mathbb D}\partial_\zeta G(\zeta, z)\frac{\overline{\omega'(\zeta)}}{\omega'(\zeta)}\partial(T_0 F_{n-1}(\zeta))d\mu(\zeta)=\\
=\int_{\mathbb D}\frac{\overline{\omega'(\zeta)}}{\omega'(\zeta)}\partial(T_0 F_{n-1}(\zeta))d\mu(\zeta)\partial_\zeta
\int_{\mathbb D} G(\zeta, z)\partial\overline\partial\phi(z)d\mu(z)=\\
=\int_{\mathbb D}\frac{\overline{\omega'(\zeta)}}{\omega'(\zeta)}\partial(T_0 F_{n-1}(\zeta))\partial\phi(\zeta)d\mu(\zeta)=
\langle(\overline{\omega'}/\omega')\partial(T_0 F_{n-1})|\partial\phi\rangle.
\end{multline*}
Это означает равенство $\partial\overline\partial F_n=-\partial[(\overline{\omega'}/\omega')\partial(T_0 F_{n-1})]$ обобщенных производных в $\mathbb D$. Из него и из
\eqref{oper-change}, в свою очередь, вытекает следующая цепочка равенств для обобщенных функций:
\begin{multline*}
|\omega'|^2\mathcal Ls_m=\sum_{n=0}^m\left(\partial\overline\partial F_n+\|T\|\partial\left(\frac{\overline{\omega'}}{\omega'}\partial(T_0 F_n)\right)\right)\|T\|^n=\\
=\partial\overline\partial F_0+\sum_{n=1}^m\left(\partial\overline\partial F_n+\partial\left(\frac{\overline{\omega'}}{\omega'}\partial(T_0 F_{n-1})\right)\right)\|T\|^n+
\partial\left(\frac{\overline{\omega'}}{\omega'}\partial(T_0 F_m)\right)\|T\|^{m+1}=\\
=\partial\left(\frac{\overline{\omega'}}{\omega'}\partial(T_0 F_m)\right)\|T\|^{m+1},
\end{multline*}
т.е. для любой функции $\varphi\in C_0^2(\varOmega)$, используя функцию $\phi:=\varphi\circ\omega\in C_0^2(\mathbb D)$, можно записать
\begin{equation*}
\langle\mathcal L s_m\,|\,\varphi\rangle=\langle\partial\left[(\overline{\omega'}/\omega')\partial(T_0 F_m)\right]\,|\,\phi\rangle\cdot\|T\|^{m+1}=
-\langle\partial(T_0 F_m)\,|\,(\overline{\omega'}/\omega')\partial\phi\rangle\cdot\|T\|^{m+1}.
\end{equation*}
Но тогда
\begin{multline*}
\langle\mathcal Lf\,|\,\varphi\rangle:=\langle f|\mathcal L\varphi\rangle=\langle f-s_m|\mathcal L\varphi\rangle+\langle s_m|\mathcal L\varphi\rangle=\\
=\langle F-S_m|\mathcal M\phi\rangle-\langle\partial(T_0 F_m)|(\overline{\omega'}/\omega')\partial\phi\rangle\cdot\|T\|^{m+1}.
\end{multline*}
Положим $1/p+1/q=1$. Применяя неравенство Гельдера и принимая во внимание соотношения \eqref{der-Fn-est} и \eqref{ser-conv}, получаем
\begin{multline*}
\left|\langle\mathcal Lf\,|\,\varphi\rangle\right|\leqslant \|F-S_m\|_{L_p(\mathbb D)}\cdot\|\mathcal M\phi\|_{L_q(\mathbb D)}+
\|DF_m\|_{L_p(\mathbb D)}\cdot\|\partial\phi\|_{L_q(\mathbb D)}\cdot\|T\|^{m+1}\leqslant\\
\leqslant\|F-S_m\|_{L_p(\mathbb D)}\cdot\|\mathcal M\phi\|_{L_q(\mathbb D)}+\|DF_0\|_{L_p(\mathbb D)}\cdot
\|\partial\phi\|_{L_q(\mathbb D)}\cdot\|\mathcal{DK}\|_p^m\cdot\|T\|^{m+1}\to 0
\end{multline*}
при $m\to\infty$ и $\|\mathcal{DK}\|_p\cdot\|T\|<1$. Таким образом,$\langle\mathcal Lf\,|\,\varphi\rangle=0$, т.е. функция $f$ удовлетворяет
уравнению $\mathcal Lf=0$ в $\varOmega$ в обобщенном смысле. В силу эллиптичности этого уравнения, оно, согласно лемме Вейля, выполняется и в классическом смысле.

\textbf{Шаг 3.} Остается показать, что $f|_\Gamma=h$. Из оценок \eqref{der-Fn-est} и \eqref{Fn-est} следует, что $F_n\in W_p^1(\mathbb D)$.
По теореме вложения Соболева, при $p>2$ отсюда вытекает, что $F_n\in C(\overline{\mathbb D})$. Так как функция $F_0$ является гармоническим продолжением граничной функции $H\in C^\alpha(\mathbb T)$, то $F_0|_{\mathbb T}=H$. Остальные функции $F_n$, вычисляемые по второй формуле из \eqref{induct-oper}, обращаются в ноль на $\mathbb T$: это можно показать, приблизив функцию
$(\overline{\omega'}/\omega')\partial(T_0 F_{n-1})\in L_p(\mathbb D)$, $p\in(1, \infty)$, при $n\geqslant 1$ финитными функциями из
$C^2_0(\mathbb D)$ и применив оценку \eqref{K-vanish}.

Таким образом, $S_m|_\mathbb T=H$ при любых $m$. Из компактности вложения $W_p^1(\mathbb D)\subset C(\overline{\mathbb D})$ и из сходимости \eqref{ser-conv} вытекает равномерная сходимость $\|F-S_m\|_{C(\overline{\mathbb D})}\to 0$, так что $F|_\mathbb T=S_m|_\mathbb T=H$. Следовательно, $f|_\Gamma=h$.

Все приведенные рассуждения справедливы при выполнении неравенств $2<p<(2(1-\alpha))^{-1}$, которые совместимы, в виду того, что принято $\alpha\in(1/2, 1)$. Поскольку $\|\mathcal{DK}\|_p\to 1$ при $p\to 2$, то для любого значения $\|T\|<1$ можно подобрать такое достаточно близкое к $2$ значение $p$, при котором $\|\mathcal{DK}\|_p\cdot\|T\|<1$. Теорема доказана.
\end{proof}

\bigskip
\noindent Астамур Олегович Багапш,\\
ФИЦ ИУ РАН,\\
ул. Вавилова, д. 44, корп. 2,\\
11933 Москва, Россия\\

\noindent Санкт-Петербургский государственный университет,\\
14 линия В.О., д. 29б,\\
199178, Санкт-Петербург, Россия\\
E-mail: a.bagapsh@gmail.com

\bigskip
\noindent Astamur Olegovich Bagapsh  ,\\
Federal Research Center ''Computer Science and Control'',\\
Vavilova str., 44, bld. 2,\\
11933 Moscow, Russia\\

\noindent Saint-Petersburg State University,\\
14 Line Vasilievskiy island, 29b,\\
199178, Saint-Petersburg, Russia\\
E-mail: a.bagapsh@gmail.com

\end{document}